\documentclass[a4paper,11pt]{amsart}
\usepackage{amssymb,amsfonts,amsmath}

\def\C{\mathbb{C}}
\def\c2{\mathbb{C}^2}
\def\R{\mathbb{R}}

\def\N{\mathbb{N}}

\def\P{\mathbb{P}}

\def\1{\bold{1}}

\def\a{\alpha}
\def\b{\beta}
\def\e{\varepsilon}

\def\f{\varphi}
\def\g{\gamma}

\def\p{\psi}

\def\cP{\mathcal P (X,\theta)}
 
 \def\cE{\mathcal E^1 (X,\theta)}
 \newcommand{\MA}{\mathrm{MA}}
\newcommand{\Amp}{\mathrm{Amp}\,}

\def\T{{\mathcal{T}}}

\def\cP{{PSH}}
\def\vol{\rm{vol}}

\def\ca{\rm{Cap}}

\newtheorem{lem}{Lemma}[section]
\newtheorem{prop}[lem]{Proposition}
\newtheorem{defi}[lem]{Definition}
\newtheorem{def/not}[lem]{Definition/Notations}

\newtheorem{thm}[lem]{Theorem}

\newtheorem{rem}[lem]{Remark}

\newtheorem{exa}[lem]{Example}

\begin{document}

\title[Stability in big cohomology classes]
{Stability of solutions to complex Monge-Amp\` ere equations in big cohomology classes}

\author{Vincent GUEDJ* and Ahmed ZERIAHI*}

 \date{\today \\ *both authors are partially supported by the ANR project MACK}
 
 \address{I.M.T., Universit{\'e} Paul Sabatier\\
31062 Toulouse cedex 09\\
France}
\email{vincent.guedj@math.univ-toulouse.fr}

\email{ahmed.zeriahi@math.univ-toulouse.fr}

\begin{abstract}
We establish various stability results for solutions of complex Monge-Amp\`ere equations in big cohomology classes, 
generalizing results that were known to hold  in the context of K\"ahler classes.
\end{abstract}

\maketitle

\section*{Introduction}

Let $(X,\omega)$ be a compact K\"ahler manifold of complex dimension $n \in \N^*$. Recall that a $(1,1)$-cohomology class  is {\it big} if it contains a {\it K\"ahler current},
i.e. a positive closed current  which dominates a K\"ahler form.  Fix $\a \in H^{1,1}(X,\R)$ a  big class and 
$\mu$ a non-negative Radon measure whose total mass $\mu(X)$ equals
$\vol(\a)$, the volume of $\a$.

The systematic study of complex Monge-Amp\`ere equations in big cohomology classes
has been initiated in \cite{BEGZ}. It has been show there that there exists a unique positive closed current $T_{\mu} \in \a$ with full Monge-Amp\`ere mass
such that 
$$
T_{\mu}^n=\mu
$$ 
if and only if $\mu$ does not charge pluripolar sets.

The purpose of this note is to study the stability properties of the solution $T_{\mu}$ to this complex Monge-Amp\`ere equation, i.e. to study the continuity properties of the mapping
$$
\mu \mapsto T_{\mu}.
$$

We can not expect this mapping to be continuous for the weakest topologies, i.e. when  the set of non pluripolar measures (resp. the set of positive currents
with full Monge-Amp\`ere masses)
 is endowed with the weak topology of Radon measures (resp. of positive currents), as the Monge-Amp\`ere operator $T \mapsto T^n$
 is not continuous either for this weak topology (this observation was made, in a local context, by Cegrell and Kolodziej in \cite{CK94}).
 On the other hand we have the following:
 
 \medskip
 \noindent {\bf PROPOSITION A.}
 {\it 
 Let $\mu_j,\mu$ be non pluripolar measures with total mass $\mu_j(X)=\mu(X)=\vol(\a)$. If $||\mu_j-\mu|| \rightarrow 0$ , then 
 $$
 T_{\mu_j} \rightarrow T_{\mu}
 \text{ in the weak sense of currents}.
 $$
 }
 \medskip
 
 Here $||\mu_j-\mu||$ denotes the total variation of the signed measure $\mu_j-\mu$.
 
 \smallskip
 
 It follows from the $dd^c$-lemma that any positive closed current $T \in \a$ decomposes as $T=\theta+dd^c \f$, for some 
 $\theta$-plurisubharmonic function $\f$. We let $\cP(X,\theta)$ denote the set of all such functions and observe that
 there is a unique $\f_{\mu} \in \cP(X,\theta)$ such that $\sup_X \f_{\mu}=0$ and $T_{\mu}=\theta+dd^c \f_{\mu}$. In the sequel 
 we let
 $$MA(\f):=\langle (\theta+dd^c \f)^n \rangle 
 $$
 denote the (non pluripolar) complex Monge-Amp\`ere measure of $\f \in PSH(X,\theta)$.
 The equation $T_{\mu}^n=\mu$ is thus equivalent to the Monge-Amp\`ere equation
 $$
 MA(\f_{\mu})=\mu.
 $$
 
 Since the weak convergence of currents $T_{\mu}$ is equivalent to the $L^1$-convergence of their normalized potentials
 $\f_{\mu}$, Proposition A can be reformulated as
 $$
\left(  || {\mu_j}-{\mu}|| \rightarrow 0 \right) \Longrightarrow \left( || \f_{\mu_j}-\f_{\mu}||_{L^1(X)} \rightarrow 0 \right).
$$

It is natural to try and estimate quantitatively how fast this convergence holds.
 Our second result yields a quantitative stability property "in energy":
 
 \medskip
 \noindent {\bf THEOREM B.}
 {\it
  There exists $C_n > 0$  such that if  $0 \geq \p, \f_1, \f_2,  \in \mathcal E^1 (X,\theta)$ are normalized by 
  $\sup_X \f_1=\sup_X \f_2$, then
 $$
\int_X \vert \f_1 - \f_2 \vert \MA (\p) \leq  C_n \cdot B^2  \cdot I (\f_1,\f_2)^{2^{-n}},
 $$
 where $B=\max \{1, \vert E (\f_1)\vert,\vert E (\f_2)\vert, \vert E (\p)\vert \}$. 
 }
 \medskip
 
 We refer the reader to the first section for the definition of the class ${\mathcal E}^1(X,\theta)$ of $\theta$-psh functions 
 $\f$ which have finite energy $E(\f)>-\infty$.  We recall here that the symmetric expression
 $$
 I(\f_1,\f_2):=\int (\f_1-\f_2) (MA(\f_2)-MA(\f_1)) \geq 0
 $$
 is used to define the important notion of "convergence in energy". Theorem B implies in particular a 
 quantitative estimate on how "convergence in energy" implies "convergence in capacity". 
 Related results were  previously obtained in \cite{BBGZ}, the latter article being a great source of inspiration for this note.
 
Let us also stress  that when the underlying cohomology class is K\"ahler, a weaker but quite elegant stability result
 was previously obtained by Blocki in \cite{Bl03}.  We briefly explain in section \ref{sec:energy} how our result can be used to derive more standard stability estimates in    this vein.
 
 \smallskip
 
 Our last result yields the strongest property of stability, assuming stronger properties on the corresponding measures.
 
 \medskip
 \noindent {\bf THEOREM C.}
 {\it  
 Assume $\mu=MA(\f_{\mu})=f_{\mu} \omega^n, \nu=MA(\f_{\nu})=f_{\nu} \omega^n$, 
where the densities $0 \leq f_{\mu},f_{\nu}$ are in   $L^p(\omega^n)$ for some $p>1$ and $\f_{\mu},\f_{\nu} \in \cP(X,\theta)$
are normalized by $\sup_X \f_{\mu}=\sup_X \f_{\nu}=0$. Then
 $$
  \Vert \f_\mu - \f_\nu \Vert_{L^{\infty} (X)} \leq M_{\tau} \Vert f_{\mu}- f_{\nu} \Vert_{L^1 (X)}^{\tau},
 $$
where  $M_{\tau} > 0$   only depends on upper bounds for the $L^p$ norms of $f_{\mu},f_{\nu}$ and
$$
\tau<\frac{1}{2^n(n+1)-1}.
$$
 }
 \medskip
 
 The existence of a unique normalized $\theta$-psh function $\f_{\mu}$ with minimal singularities such that $(\theta+dd^c \f_{\mu})^n=\mu$
 when $\mu$ has $L^p$-density, $p>1$, has been established in \cite[Theorem 4.1]{BEGZ}, generalizing Kolodziej's celebrated result \cite{Kol98}.
 
 It is likely that the exponent $\tau$ we obtain here is not sharp. When $\a$ is a K\"ahler class, a better exponent
 was obtained by Kolodziej in \cite{Kol03} and later on improved by Dinew-Zhang in \cite{DZ} (see also \cite{Hiep}
 for some other generalization).

\medskip

\noindent {\it Notations.}
In the whole article we fix 

$\bullet$ $(X,\omega)$ a compact K\"ahler manifold equipped with a K\"ahler form $\omega$,

$\bullet$ $\a \in H^{1,1}(X,\R)$ a big cohomology class,

$\bullet$ $\theta$  a smooth closed 
$(1,1)$-form representing $\a$.

\section{Preliminary results on big cohomology classes}

We briefly recall here some material  developed in full detail in \cite{BEGZ}.

\subsection{Quasi-psh functions}
Recall that an upper semi-continuous function $$\f:X\to[-\infty,+\infty[$$
is said to be \emph{$\theta$-psh} iff $\f$ is locally the sum of  a smooth and a psh function, and $\theta+dd^c\f\ge 0$ in the sense of currents, where $d^c$ is normalized so that 
$$
dd^c=\frac{i}{\pi}\partial\overline{\partial}.
$$

By the $dd^c$-lemma any closed positive $(1,1)$-current $T$ cohomologous to $\theta$ can conversely be written as $T=\theta+dd^c\f$ for some $\theta$-psh function $\f$ which is furthermore unique up to an additive constant. 

The set of all $\theta$-psh functions $\f$ on $X$ will be denoted by $\cP(X,\theta)$ and endowed with the weak topology, which coincides with the $L^1(X)$-topology. By Hartogs' lemma $\f\mapsto\sup_X\f$ is continuous in the weak topology. Since the set of closed positive currents in a fixed cohomology class is compact (in the weak topology), it follows that the set of $\f\in\cP(X,\theta)$ normalized by $\sup_X\f=0$ is compact. 

\smallskip

We introduce the extremal function $V_\theta$ defined by
\begin{equation}\label{equ:extrem}V_\theta(x):=\sup\{\f(x)|\f\in\cP(X,\theta),\sup_X\f\le 0\}.
\end{equation}
It is a $\theta$-psh function with \emph{minimal singularities} in the sense of Demailly, i.e.~we have 
$\f\le V_\theta+O(1)$ for any $\theta$-psh function $\f$. In fact it is straightforward to see that the following 'tautological maximum principle' holds:
\begin{equation}\label{equ:max}\sup_X\f=\sup_X(\f-V_\theta)
\end{equation}
for any $\f\in\cP(X,\theta)$.

\subsection{Ample locus and regularity of envelopes}

The cohomology class $\a=\{\theta\}\in H^{1,1}(X,\R)$ is said to be \emph{big} iff there exists a closed $(1,1)$-current 
$$
T_+=\theta+dd^c\f_+
$$ 
cohomologous to $\theta$ such that $T_+$ is \emph{strictly positive} (i.e. $T_+\ge \e_0 \omega$ for some $\e_0>0$). 
By Demailly's regularisation theorem~\cite{Dem92}  one can then furthermore assume that $T_+$ has \emph{analytic singularities}, that is there exists $c>0$ such that locally on $X$ we have 
$$
\f_+=c\log\sum_{j=1}^N|f_j|^2\text{ mod }C^\infty
$$
where $f_1,...,f_N$ are local holomorphic functions. Such a current $T$ is then smooth on a Zariski open subset $\Omega$, and 
the \emph{ample locus}  $\Amp(\a)$ of $\a$ is defined as the largest such Zariski open subset (which exists by the Noetherian property of closed analytic subsets). 

Note that \emph{any} $\theta$-psh function $\f$ with minimal singularities is locally bounded on the ample locus $\Amp(\a)$ since it has to satisfy $\f_+\le\f+O(1)$. 
Note that $\f_+$ does not have minimal singularities unless $\a$ is a K\"ahler class.

\smallskip

In case $\a=\{\theta\}\in H^{1,1}(X,\R)$ is a \emph{K{\"a}hler} class, plenty of \emph{smooth} $\theta$-psh functions are available. 
When $\a$ is both big and nef (i.e. $\a$ belongs to the closure of the cone of K\"ahler classes), a good regularity theory is available thanks to \cite{BEGZ}.
However for a general \emph{big} class the existence of even a \emph{single} $\theta$-psh function with minimal singularities that is also $C^\infty$ 
on the ample locus $\Amp(\a)$ is unknown.  

On the other hand  we have the following  regularity result of Berman-Demailly on the ample locus ~\cite{BD}: 

\begin{thm}\label{thm:c11} 
The function $V_\theta$ has locally bounded Laplacian on $\Amp(\theta)$. 

In particular the Monge-Amp{\`e}re measure $\MA(V_\theta)$ has $L^\infty$-density with respect to Lebesgue measure. More specifically we have $\theta\ge 0$ pointwise on $\{V_\theta=0\}$ and 
$$
\MA(V_\theta)={\bf 1}_{\{V_\theta=0\}}\theta^n.
$$
\end{thm}

Since $V_\theta$ is quasi-psh this result is equivalent to the fact that the curent $\theta+dd^c V_\theta$
has $L^\infty_{loc}$ coefficients on $\Amp(\a)$ and shows in particular by Schauder's elliptic estimates that $V_\theta$ is in fact $C^{2-\e}$ on $\Amp(\a)$ for each $\e>0$.

\subsection{Full Monge-Amp\`ere mass}

In~\cite{BEGZ} the \emph{non-pluripolar product} 
$$
(T_1,...,T_p)\mapsto\langle T_1\wedge...\wedge T_p\rangle
$$ 
of closed positive $(1,1)$-currents is shown to be well-defined as a closed positive $(p,p)$-current putting no mass on pluripolar sets. 
In particular given $\f_1,...,\f_n\in\cP(X,\theta)$ we define their mixed Monge-Amp{\`e}re measure as 
$$\MA(\f_1,...,\f_n)=\langle(\theta+dd^c\f_1)\wedge...\wedge(\theta+dd^c\f_n)\rangle.$$
It is a non-pluripolar positive measure whose total mass satisfies 
$$
\int_X\MA(\f_1,...,\f_n)\le \vol(\a)
$$
where the right-hand side denotes the \emph{volume} of the cohomology class $\a$. 
If $\f_1,...,\f_n$ have minimal singularities then they are locally bounded on $\Amp(\a)$, and the product 
$$
(\theta+dd^c\f_1)\wedge...\wedge(\theta+dd^c\f_n)
$$ 
is thus well-defined by Bedford-Taylor~\cite{BT82}. Its trivial extension to $X$ coincides with $\MA(\f_1,...,\f_n)$, and we have 
$$
\int_X\MA(\f_1,...,\f_n)=\vol(\a).
$$
In case $\f_1=...=\f_n=\f$, we simply set
$$
\MA(\f)=\MA(\f,...,\f)
$$
and say that $\f$ has \emph{full Monge-Amp{\`e}re mass} iff $\int_X\MA(\f)=\vol(\a)$. We let
$$
{\mathcal E}(X,\theta):=\left\{ \f \in PSH(X,\theta) \, | \, \int_X\MA(\f)=\vol(\a) \right\}
$$
denote the set of $\theta$-psh functions with full Monge-Amp\`ere mass.
We thus see that $\theta$-psh functions with minimal singularities have full Monge-Amp{\`e}re mass, but the converse is not true. 

A crucial point is that the non-pluripolar Monge-Amp{\`e}re operator is continuous along monotonic sequences of functions with full Monge-Amp{\`e}re mass. 
In fact we have (cf.~\cite{BEGZ} Theorem 2.17):

\begin{prop}\label{prop:cont} 
The operator
$$
(\f_1,...,\f_n)\mapsto\MA(\f_1,...,\f_n)
$$
is continuous along monotonic sequences of functions with full Monge-Amp{\`e}re mass. 
If $\int_X(\f-V_\theta)\MA(\f)$ is finite, then 
$$
\lim_{j\to\infty}(\f_j-V_\theta)\MA(\f_j)=(\f-V_\theta)\MA(\f)
$$
for any monotonic sequence $\f_j\to\f$.
\end{prop}

\subsection{Weighted energies}
 
Let $\p \in PSH(X,\theta)$ be a $\theta$-psh function with minimal singularities. Its  \emph{Aubin-Mabuchi energy} is
$$
E(\p):=\frac{1}{n+1}\sum_{j=0}^n\int_X(\p-V_\theta) \langle (\theta+dd^c \p)^j \wedge (\theta+dd^c V_{\theta})^{n-j} \rangle.
$$
One can check \cite{BEGZ} that its G{\^a}teaux derivatives are given by
$$
E'(\p)\cdot v=\int_X v \MA(\p)
$$
showing in particular that $E$ is non-decreasing. 

\begin{defi}
We let  ${\mathcal E}^1(X,\theta)$ denote  the class of all $\theta$-plurisubharmonic functions $\f$ such that
$$
E(\f):=\inf_{\p \geq \f} E(\p)>-\infty
$$
where the infimum is taken over all functions $\p$ with minimal singularities.
\end{defi}

Alternatively a function $\f$ belongs to ${\mathcal E}^1(X,\theta)$ if and only if it belongs to ${\mathcal E}(X,\theta)$ and
$\f \in MA(\f)$. 

More generally, given $\chi= \R \rightarrow \R$ a convex increasing function such that $\chi(-\infty)=-\infty$,
one considers, for $\p$ with minimal singularities,
$$
E_{\chi}(\p):=\frac{1}{n+1}\sum_{j=0}^n\int_X \chi(\p-V_\theta) \langle (\theta+dd^c \p)^j \wedge (\theta+dd^c V_{\theta})^{n-j} \rangle.
$$
This weighted energy is again non-decreasing \cite[Proposition 2.8]{BEGZ}, hence the following:

\begin{defi}
We let  ${\mathcal E}_{\chi}(X,\theta)$ denote  the class of all $\theta$-plurisubharmonic functions $\f$ such that
$$
E_{\chi}(\f):=\inf_{\p \geq \f} E_{\chi}(\p)>-\infty
$$
where the infimum is taken over all functions $\p$ with minimal singularities.
\end{defi}

One can easily check that these classes exhaust the class of functions with full Monge-Amp\`ere mass \cite[Proposition 2.11]{BEGZ},
$$
 {\mathcal E}(X,\theta) =\bigcup_{\chi} {\mathcal E}_{\chi}(X,\theta).
 $$

We finally introduce the \emph{symmetric} expression
$$
I(\f,\p):=\int_X(\f-\p)(\MA(\p)-\MA(\f)) \geq 0
$$
where the non-negativity can be deduced from the following formula
\begin{equation}
\label{equ:I}
I(\f,\p)=\sum_{j=0}^{n-1}\int_\Omega d(\f-\p)\wedge d^c(\f-\p)\wedge \langle (\theta+dd^c\f)^j\wedge(\theta+dd^c\psi)^{n-1-j} \rangle.
\end{equation}
 
\begin{defi}
 A sequence of functions $\f_j \in \cE$  \emph{converges in energy} towards $\f \in\cE$ if $I(\f_j,\f)\to 0$ as $j\to\infty$. 
\end{defi}

 This notion is introduced in \cite{BBGZ} where it is shown that convergence in energy  implies continuity of the complex Monge-Amp\`ere operator.

\subsection{Monge-Amp{\`e}re capacity}

  As in~\cite{GZ05, BEGZ} we define the \emph{Monge-Amp{\`e}re (pre)capacity} in our setting as the upper envelope of all measures 
  $\MA(\f)$ with $\f\in\cP(X,\theta)$, $V_\theta-1\le\f\le V_\theta$, i.e.
\begin{equation}\label{equ:cap}
\ca(B):=\sup \left\{\int_B\MA(\f),\,\f\in\cP(X,\theta),\,V_\theta-1\le\f\le V_\theta\text{ on }X \right\}.
\end{equation}
for every Borel subset $B$ of $X$.

By definition, a positive measure $\mu$ is absolutely continuous with respect the capacity $\ca$ iff $\ca(B)=0$ implies $\mu(B)=0$. This means exactly that $\mu$ is non-pluripolar in the sense that $\mu$ puts no mass on pluripolar sets. Since $\mu$ is subadditive, it is in turn equivalent to the existence of a non-decreasing right-continuous function $F:\R_+\to\R_+$ such that 
$$\mu(B)\le F(\ca(B))$$
for all Borel sets $B$. Roughly speaking the speed at which $F(t)\to 0$ as $t\to 0$ measures "how non-pluripolar" $\mu$ is. 

 \begin{defi}
Fix $\b>0$.
 We say that $\mu$ satisfies the condition ${\mathcal H}(\b)$ if there exists $C_{\b}>0$ such that
 for all Borel sets $B \subset X$,
 $$
 \mu(B) \leq C_{\b} \ca(B)^{\b+1}.
 $$
 If this holds for all $\b>0$, we say that $\mu$ satisfies the condition ${\mathcal H}(\infty)$.
 \end{defi}
 
 Such conditions were introduced by Kolodziej in \cite{Kol98} who showed that measures $\mu=MA(\f)$ satisfying the condition
 ${\mathcal H}(\b)$ are such that $\f$ is continuous if the cohomology class $\a$ is K\"ahler. He further observed that
 if $\mu=f \omega^n$ has density in $L^p$ for some $p>1$, then $\mu$ satsifies condition ${\mathcal H}(\infty)$.
 
 These results were later on extended
 to the case of big cohomology classes in \cite{EGZ,BEGZ,EGZ11}.

 \smallskip
 
 Recall that the complex Monge-Amp\`ere operator $\f \mapsto MA(\f)$ is discontinuous for the $L^1$-topology.
 One needs to require a stronger notion of convergence of potentials:
 
  \begin{defi}
 A sequence $(\f_j)$ of $\theta$-plurisubharmonic functions converges in capacity towards $\f$ if for all $\e>0$,
 $$
 \ca\left(\{�|\f_j-\f| >\e \}\right) \rightarrow 0
 \text{ as } j \rightarrow +\infty.
 $$
 \end{defi}
 
 If a  sequence $\f_j \in {\mathcal E}^1(X,\theta)$ converges to $\f \in {\mathcal E}^1(X,\theta)$ in capacity, then
 $MA(\f_j)$ weakly converges towards $MA(\f)$ \cite{GZ07,DH}. This generalizes previous continuity statements, as monotonic
 convergence implies convergence in capacity.

\section{Weak stability properties}

In this section we establish the weakest stability property, i.e. Proposition A stated in the introduction.

 \subsection{Unstability}
 
 We start by observing that one can not expect stability in general. Recall \cite{BEGZ} that if $\mu$ is  a non-negative Radon measure which vanishes
 on pluripolar sets and whose total mass equals $\vol(\a)$,  then there exists a unique positive closed current $T_{\mu} \in \a$
 with full Monge-Amp\`ere mass and such that 
 $$
 \langle T_{\mu}^n \rangle=\mu.
 $$
 The current $T_{\mu}$ decomposes as $T_{\mu}=\theta+dd^c \f_{\mu}$, where $\f_{\mu} \in \cP(X,\theta)$ is uniquely determined, once normalized
 by $\sup_X \f_{\mu}=0$.
 
 One can not expect the operator $\mu \mapsto \f_{\mu}$ (or equivalently $\mu \mapsto T_{\mu}$) to be continuous, as its inverse
 operator $\f \mapsto \langle (\theta+dd^c \f)^n \rangle$ is not either. Here is a variation on a classical local example \cite{Ceg83} of such discontinuous behavior:
 
 \begin{exa}
 The functions
 $$
 \p_j(z_1,z_2):=\frac{1}{2j} \log \left[ |z_1^j+z_2^j|^2+1 \right]
 $$
 are smooth and plurisubharmonic in $\C^2$. They form a locally bounded sequence which converges in $L_{loc}^1(\C^2)$ towards
 $$
 \p(z_1,z_2)=\log \max[1, |z_1|, |z_2| ].
 $$
 Observe that the Monge-Amp\`ere measures $(dd^c \p_j)^2$ vanish identically, while $(dd^c \p)^2$ is the Lebesgue measure on the
 real torus $\{ |z_1|=|z_2|=1 \}$.

 One can globalize this example, working on $X=\C\P^2$ equipped with its Fubini-Study K\"ahler form $\theta=\omega_{FS}$. Set
 $$
 \f_j[z]=\frac{1}{2j} \log \left[ |z_1^j+z_2^j|^2+|z_0|^{2j} \right] -\log ||z||,
 $$
 where $[z]=[z_0:z_1:z_2]$ denotes the homogeneous coordinates in $\C\P^2$ and $(z_0=0)$ denotes the hyperplane at infinity,
 $\C\P^2=\C^2 \cup (z_0=0)$.
 
 The functions $\f_j$ are $\theta$-psh and smooth in $\C\P^2 \setminus S_j$, where 
 $S_j$ denotes the finite set of points at infinity $\{z_0=0=z_1^j+z_2^j\}$.
 The $\f_j$'s converge in $L^1(\C\P^2)$ towards
 $$
  \f(z_1,z_2)=\log \max[|z_0|, |z_1|, |z_2| ] -\log ||z||,
  $$
  whose Monge-Amp\`ere measure is again the Lebesgue measure on the torus. 
 \end{exa}
 
 This example is not so satisfactory since the Monge-Amp\`ere measures $MA(\f_j)$ are all supported on the
 (pluripolar) hyperplane at infinity. We thus propose a slightly more elaborate construction where the approximants
 are uniformly bounded:

  \begin{exa}
 Using the same notations as in previous example, we set
 $$
 \Phi_j:=\log \left[ e^{\f_j}+e^{-K} \right],
 $$
 where $K>0$ is a large constant. The reader will easily check that  
 $$
 \theta+dd^c \Phi_j=\frac{e^{\f_j} \theta_{\f_j}+e^{-K} \theta}{e^{\f_j} +e^{-K}}+\frac{e^{\f_j-K} d \f_j \wedge d^c \f_j}{\left[ e^{\f_j}+e^{-K}\right]^2} \geq 0,
 $$
 so that $\Phi_j$ are uniformly bounded $\theta$-psh functions on $\C\P^2$. We use here the shortcuts $\theta=\omega_{FS}$
 and  $\theta_u:=\theta+dd^c u$.
 
 A similar computation can be made for $\Phi:=\log \left[ e^{\f}+e^{-K} \right]$, showing in particular that
$$
MA(\Phi) \geq \frac{e^{ -2 \sqrt{3}}}{\left[ e^{-\sqrt{3}}+e^{-K} \right]^2} \sigma_{\T}
$$
dominates a multiple of the (normalized) Lebesgue measure $\sigma_{\T}$ on the real torus $\T=\{|z_0|=|z_1|=|z_2|\}$.

This multiple can be made arbitrarily close to $1$ by choosing $K$ large enough. On the other hand $MA(\Phi_j)$ can be computed explicitly
by using that $(dd^c \p_j)^2, dd^c \p_j \wedge d \p_j, dd^c \p_j \wedge d^c \p_j$ are all zero in $\C^2$. One can this way verify that
any cluster point of $MA(\Phi_j)$ is different from $MA(\Phi)$, although $\Phi_j$ converges towards $\Phi$.
 \end{exa}

 \subsection{Proof of Proposition A}
  
 We now prove a qualitative property of stability under a weak domination assumption.
 Let $\mu_j,\mu$ be non negative Radon measures on $X$ which do not charge pluripolar sets
 and whose total mass equals $\vol(\a)$. 
 
 \medskip
 \noindent {\bf PROPOSITION A'.}
 {\it 
 If the measures $\mu_j=f_j \nu$ are all absolutely continuous with respect to a  fixed non pluripolar measure $\nu$ and 
 $f_j \rightarrow f$ in $L^1(\nu)$, then 
 $$
 T_{\mu_j} \rightarrow T_{\mu}
 \text{ in the weak sense of currents},
 $$
 where $\mu=f \nu$.
 }
 \medskip
 
 This result can be seen as a generalization of a local result of Cegrell-Kolodziej \cite{CK06} who
 asked for $f_j$ to be uniformly bounded. 
 
 \medskip
 
 \noindent {\it Proof.} We let $\f_j,\f$ denote the normalized Monge-Amp\`ere potentials,
 $$
 \mu_j=(\theta+dd^c \f_j)^n, \,
 \mu=(\theta+dd^c \f)^n,
 \text{ with } \sup_X \f_j=\sup_X \f=0.
 $$
 We assume  that 
 $
 \mu_j=f_j \nu, \, \mu=f \nu,
 $
 where $\nu$ vanishes on pluripolar sets and $f_j \rightarrow f$ in $L^1(\nu)$, and we are going to show that in this case
 $(\f_j)$ converges in $L^1(X)$ towards $\f$.
 
 By weak compactness, we can assume -up to extracting- that $\f_j \rightarrow \p \in \cP(X,\theta)$, with $\sup_X \p=0$.
 Extracting again, we can also assume that there exists $g \in L^1(\nu)$ such that
 $$
 f_j \leq g \text{ for all } j \in \N.
 $$
 Since the measure $g \nu$ does not charge pluripolar sets, it follows from \cite[Proposition 3.2]{BEGZ} 
 that there exist $\chi:\R \rightarrow \R$ a convex increasing weight and $C>0$ such that $\chi(-\infty)=-\infty$ and 
 for all $j \in \N$,
 $$
 \int (-\chi) (\f_j-V_{\theta}) g d\nu \leq C.
 $$
 This shows that 
 $$
\int (-\chi) (\f_j-V_{\theta}) MA(\f_j) \leq C,
$$
hence \cite[Proposition 2.10]{BEGZ} insures that $\p \in {\mathcal E}_{\chi}(X,\theta)$.

The functions $\p_j:=(\sup_{l \geq j} \f_l)^* \in \cP(X,\theta)$ decrease to $\p$ and satisfy
$$
MA(\p_j) \geq (\inf_{l \geq j} f_l) \nu.
$$
We infer $MA(\p) \geq \mu=f \nu$, whence equality since these measures have the same mass $\vol(\a)$.

This shows that $MA(\p)=MA(\f)$, hence these normalized potentials have to be equal,  by the uniqueness in \cite[Theorem 3.1]{BEGZ}.
$\Box$
 
 \medskip
 
 We finally observe that Proposition A and Proposition A' are equivalent. Indeed if $\mu_j=f_j\nu$ and $\mu=f \nu$, then
by definition
$$
||\mu_j-\mu||=||f_j-f||_{L^1(\nu)}
$$
so that Proposition A' is a particular case of Proposition A. 

Conversely, if $\mu_j,\mu$ are non pluripolar measures of the same mass $\vol(\a)$ such that
$||\mu_j-\mu|| \rightarrow 0$, then
$$
\nu:=\mu+\sum_{j \geq 0} 2^{-j} \mu_j
$$
is  a well defined non pluripolar Radon measure with respect to which $\mu_j,\mu$ are absolutely continuous, thus the
hypotheses of Proposition A' are satisfied.

 \section{Stability in energy} \label{sec:energy}
 
 \subsection{Case of a K\"ahler class}
 
  Our starting point is the following result which is  a refinement of \cite[Lemma 3.12]{BBGZ}:
  
  \begin{lem} \label{lem:BBGZ}
  There exists $\kappa_n > 0$  such that if
  $0 \geq \f_1, \f_2, \p_1, \p_2 \in \cE$ satisfy $E (\f_i) \geq - B$, $E (\p_i) \geq - B$, then
  \begin{equation}
  \label{eq:P0}
  \left\vert \int (\f_1 - \f_2) (\MA (\p_1) - \MA (\p_2)) \right\vert \leq  \kappa_n B_+^{2} I (\f_1,\f_2)^{ 2^{-n}} I (\p_1,\p_2)^{2^{ - n}}
  \end{equation}
  and
 \begin{equation}
  \label{eq:P1}
  \int d (\f_1 - \f_2) \wedge d^c (\f_1 - \f_2) \wedge T_{n - 1}  \leq \kappa_n B_+^{2}  I (\f_1,\f_2)^{2^{-(n-1)}},
 \end{equation}
  where $B_+ :=\max(1,B)$ and
  $$
 T_{n-1}:= \sum_{j = 0}^{n - 1} (\theta+dd^c\p_1)^j\wedge(\theta+dd^c\p_2)^{n-1-j}.
  $$
  \end{lem}
  
 A particular case of the second inequality was obtained in \cite{GZ07}  when $\a$ is a K\"ahler class (see also \cite{Bl03} for bounded functions).

  \begin{proof} 
  Observe that the first inequality follows from the second one using Stokes formula and Cauchy-Schwarz inequality.  
  We also note that it suffices to establish (\ref{eq:P1}) when $\p_1=\p_2=:\p$, the general case follows
  by considering $\p=(\p_1+\p_2)/2$.
  
  Set $u:=\f_1-\f_2$, $v:=(\f_1+\f_2)/2$ and for each $p=0,...,n-1$,
$$
b_p:=\int_X du\wedge d^c u\wedge\theta_v^{p}\wedge\theta_{\psi}^{n-p-1},
$$
where $\theta_v:=\theta+dd^c v$.
Our goal is to bound $b_0$ from above, since
$$
b_0=\frac{1}{n} \int d(\f_1-\f_2) \wedge d^c (\f_1-\f_2) \wedge T_{n-1},
$$
as $\p=\p_1=\p_2$.

Using Stokes theorem we obtain
\begin{eqnarray*}
b_p &=& \int_X du\wedge d^c u\wedge\theta_v^{p+1}\wedge\theta_{\psi}^{n-p-2}
+\int_X du\wedge d^c u\wedge dd^c(\psi-v)\wedge\theta_v^p\wedge\theta_\psi^{n-p-2} \\
&=& b_{p+1}-\int_X du\wedge d^c(\psi-v)\wedge dd^c u\wedge\theta_v^p\wedge\theta_{\psi}^{n-p-2} \\
& = & b_{p+1}-\int_X du\wedge d^c(\psi-v)\wedge\theta_{\f_1}\wedge\theta_v^p\wedge\theta_{\psi}^{n-p-2} \\
& +& \int_X du\wedge d^c(\psi-v)\wedge\theta_{\f_2}\wedge\theta_v^p\wedge\theta_{\psi}^{n-p-2}. 
\end{eqnarray*}
noting that $dd^c u = \theta_{\f_1}- \theta_{\f_2}$.

Recall that $\theta_{\f_i}\le 2\theta_v$, hence Cauchy-Schwarz inequality and (\ref{equ:I}) yield
$$
\left|\int_X du\wedge d^c(\psi-v)\wedge\theta_{\f_i}\wedge\theta_v^p
\wedge\theta_{\psi}^{n-p-2}\right|
\le 2 b_{p+1}^{1/2}I(\psi,v)^{1/2}.
$$

It follows from \cite[Lemma 2.7]{BBGZ} that $I(\p,v) \leq a_n B_+$, where $a_n > 1$ is a uniform constant, thus
\begin{equation}\label{equ:boundb}
b_p \le b_{p+1}+ 2 (a_n B_+)^{1 \slash 2} \sqrt{b_{p+1}}=h(b_{p+1}),
\end{equation}
where $h (t) := t + 2 (a_n B_+)^{1 \slash 2} \sqrt{t},$ for $t \geq 0$, is monotone increasing in $t$. Thus
$$
b_0 \leq h^{n-1} (b_{n-1}) \leq h^{n-1} (I(\f_1,\f_2)),
$$
since
$$
b_{n-1} \leq \sum_{j=0}^{n-1} \int du \wedge d^c u \wedge \theta_{\f_1}^j \wedge \theta_{\f_2}^{n-1-j}=I(\f_1,\f_2).
$$
Here $h^{n-1}:=h \circ \cdots \circ h$ denotes the $(n-1)^{th}$-iterate of the function $h$.

 Observe that $h (t) \leq C_1 \sqrt{t}$ for $0 \leq t\leq 1$, where $C_1 := 1 + 2 (a_n B_+)^{1 \slash 2}$. 
We infer that if $0 \leq t \leq C_1^{-2^n}$ then $h^{n-1} (t) \leq C_1^2 t^{2^{-(n-1)}}$. Therefore
$$
b_0 \leq C_1^2 I(\f_1,\f_2)^{2^{-(n-1)}}
\text{ if } I(\f_1,\f_2) \leq C_1^{-2^n}.
$$

When $I(\f_1,\f_2)$ is relatively large, i.e. when $I(\f_1,\f_2)> C_1^{-2^n}$, we use
\cite[Lemma 2.7]{BBGZ} again to bound from above $b_0 \leq a_n B_+$, thus obtaining
$$
b_0 \leq a_n B_+ C_1^{2} I(\f_1,\f_2)^{2^{-(n-1)}}.
$$
In both cases we can bound from above $b_0$ by $\kappa_n B_+^2$.
 \end{proof}

  When the underlying cohomology class $\a$ is K\"ahler, one can use the classical Poincar\'e inequality to deduce from Lemma \ref{lem:BBGZ}
  a quantitative stability inequality.
 Indeed assume that $\theta = \omega$ is a K\"ahler form on $X$ and, for simplicity, that $MA(\f_i)=f_i \omega^n$ are
 absolutely continuous with respect to Lebesgue measure, with $L^2$-densities.
 
We can apply the inequality (\ref{eq:P1}) with $\p_1=\p_2=0$ and 
obtain a gradient estimate in terms of the energy deviation: for any $\f_1, \f_2 \in \mathcal{E}^1 (X,\omega)$ satisfying 
  $E (\f_i) \geq - B$,
$$
\int_X d (\f_1 - \f_2) \wedge d^c (\f_1 - \f_2) \wedge \omega^ {n - 1} \leq \kappa_n B_+^2 I (\f_1,\f_2)^{1 \slash 2^{n-1}},
$$
where
$$
I(\f_1,\f_2)=\int (\f_1-\f_2)(f_2-f_1) \omega^n \leq ||\f_1-\f_2||_{L^2} ||f_1-f_2||_{L^2}
$$
if $MA(\f_i)=f_i \omega^n$ have $L^2$-densities.

We normalize the potentials $\f_i$  so that $\sup_X \f_1=\sup_X \f_2=0$. It follows then from elementary arguments
(see \cite{GZ07}) that the energies of the $\f_i$'s are uniformly bounded from below, since
$$
\int (-\f_i) MA(\f_i)=\int (-\f_i)  f_i \omega^n \leq ||\f_i||_{L^2} ||f_i||_{L^2},
$$
while Poincar\'e's inequality yields
$$ 
\Vert \f_1 - \f_2\Vert_{L^2 (X)}^2   \leq \delta_n \int_X d (\f_1 - \f_2) \wedge d^c (\f_1 - \f_2) \wedge \omega^ {n - 1},
$$
for some uniform constant $\delta_n > 0$. We have thus proved the following stability property:

\begin{prop}
Let $(X,\omega)$ be a compact K\"ahler manifold. Let $\f_1,\f_2 \in {\mathcal E}^1(X,\omega)$ be solutions of
$(\omega+dd^c \f_i)^n=f_i \omega^n$, where $\int_X f_i \omega^n=\int _X \omega^n$, $f_i \in L^2(X)$  and $\int (\f_1-\f_2) \omega^n=0$. Then
$$
\Vert \f_1 - \f_2\Vert_{L^2 (X)}   \leq C ||f_1-f_2||^{1/{(2^n -1)}}_{L^2(X)},
$$
where $C>0$ is  a uniform constant.
\end{prop}

This result can be seen as a quantitative version of Proposition A' when $\nu=\omega^n$.
Its purpose  is to illustrate, in a simple setting, how Lemma \ref{lem:BBGZ} can be used to obtain quantitative
stability properties.
As we shall see in the sequel, similar inequalities will continue to hold in more general contexts.

\subsection{The general case} 
 
 We now go back to our original situation, when the cohomology class $\{\theta\} \in H^{1,1}(X,\R)$ is merely big.
 We start by establishing an important particular case of Theorem B:
 
 \begin{prop} \label{thm:general} 
 There exists $C> 0$ such that for every $0 \geq \f, \p \in \mathcal E^1 (X,\theta)$ normalized by $\sup_X \f=\sup_X \p$, 
 $$
\Vert \f - \psi \Vert_{L^1(X)} \leq  C \cdot B^2  \cdot I (\f,\p)^{1 \slash 2^{n}},
 $$
 where $B := \max \{1, \vert E (\f)\vert, \vert E (\p)\vert \}$.
 \end{prop}

\begin{proof}
We can assume without loss of generality that $\nu=\omega^n$ is normalized so that $\nu(X):=\int_X \omega^n=\vol(\a)$.
If $\f \equiv \p$ there is nothing to prove, so we assume in the sequel that $\f \neq \p$. Reversing the roles of $\f,\p$,
we can assume that $\nu(\f<\p)>0$.

Set $Q_t:=\{x \in X \, | \, \f(x) >\p(x) -t\}$. We can find arbitrarily small $t>0$ such that $\nu(Q_t) <\vol(\a)$, otherwise
$\f \geq \p$ on $X$.  Observe also that $\nu(Q_t)>0$ for all $t>0$, otherwise $\f \leq \p-t$ contradicting
our normalizing assumption, thus for arbitrarily small $t>0$,
$$
0< a:=\frac{\nu(Q_t)}{\vol(\a)} < 1.
$$
We also set $b:=1-a=\nu(X \setminus Q_t) / \vol(\a) \in ]0,1[$ and decompose
$$
||\f-\p||_{L^1(\nu)}=\int_{Q_t}  (\f-\p) d\nu+\int_{X \setminus Q_t} (\p-\f) d\nu+O(t).
$$
We are going to bound from above each of these integrals by establishing estimates that are independent of $t$ and
then let $t$ decrease to zero.

It follows from \cite{BEGZ} that there exists uniquely determined functions $u,v \in \cP(X,\theta)$ with minimal singularities such that
$$
MA(u)=a^{-1} \1_{Q_t} \, \nu, \,
MA(v)=b^{-1} \1_{X \setminus Q_t} \, \nu
\text{ and }
\sup_X u=\sup_X v=0.
$$
 We also set
 $$
 U:=a^{1/n} u+(1-a^{1/n}) V_{\theta}
 \text{ and } 
V:=b^{1/n} v+(1-b^{1/n}) V_{\theta}.
$$
Observe that $U,V \in \cP(X,\theta)$ again have minimal singularities and are still normalized
by $\sup_X U=\sup_X V=0$ (by the tautological maximum principle). Moreover
$$
MA(U) \geq a MA(u) \text{ while } MA(V) \geq b MA(v),
$$
therefore
$$
a \int (\f-\p) MA(u)+b\int (\p-\f) MA(v) \leq \int (\f-\p) (MA(U)-MA(V)).
$$
It follows from Lemma \ref{lem:BBGZ} that the latter is bounded from above by 
$$
\kappa_n B^2 I(\f,\p)^{2^{-n}} I(U,V)^{2^{-n}},
$$
where $B=\max(1,-E(\f),-E(\p),-E(U),-E(V))$.

Since $I(U,V)$ is controlled from above if we can bound from below the energies of $U$ and $V$
(see \cite[Lemma 2.7]{BBGZ}), it remains to estimate the latter.

This is in principle very easy, as $U$ and $V$ have minimal singularities, however we want to make clear that
the corresponding bounds are independent of $t$ (i.e. independent of $a$ and $b$).  Since
$MA(u)=g \omega^n$ has density in $L^2$ (even $L^{\infty}$), It follows from \cite[Theorem 4.1]{BEGZ} that
$$
||u-V_{\theta}||_{L^{\infty}(X)} \leq c ||g||_{L^2}^{1/n} \leq c' a^{-1/n},
$$
since $g=a^{-1} \1_{Q_t} $. Therefore
$$
||U-V_{\theta}||_{L^{\infty} (X)}=a^{1/n} ||u-V_{\theta}||_{L^{\infty}(X)} \leq c''.
$$
We similarly get a uniform bound from above on $||V-V_{\theta}||_{L^{\infty} (X)}$.
Therefore
$$
-c''' \leq E(U),E(V) \leq 0,
$$
hence the proof is complete.
 \end{proof}

 We observe the following easy consequence of the previous estimates:

 \begin{lem} \label{lem:corol}
 There exists $C_n> 0$ such that for any $0 \geq \f, \p , u \in \mathcal E^1 (X,\theta)$ normalized by $\sup_X \f=\sup_X \p$,
 $$
\int_X (\f - \psi) \MA (u) \leq  C_n \cdot B^2 \cdot I (\f,\p)^{1 \slash 2^{n}},
 $$
 where $B:=\max \{1, \vert E (\f)\vert, \vert E (\p)\vert, \vert E (u)\vert \}$.
 \end{lem}
 
 \begin{proof} 
 We decompose
 $$
 \int_X (\f - \psi) \MA (u) = \int_X (\f - \psi) (\MA (u) - \MA (V_{\theta})) + \int_X (\f - \psi) \MA (V_{\theta})
 $$
 and observe that Lemma \ref{lem:BBGZ}  allows to bound from above the first term while the second once is controlled by Proposition \ref{thm:general}, since  
 $\MA (V_{\theta})$ has a bounded density with respect to $\omega^n$ by Theorem \ref{thm:c11}.
 \end{proof}

We can now prove Theorem B:
 
\begin{thm} \label{thm:B}
 There exists $C_n > 0$  such that if  $0 \geq \p, \f_1, \f_2,  \in \mathcal E^1 (X,\theta)$ are normalized by $\sup_X \f_1 = \sup_X \f_2 $, then
 $$
\int_X \vert \f_1 - \f_2 \vert \MA (\p) \leq  C_n \cdot B^2  \cdot I (\f_1,\f_2)^{2^{-n}},
 $$
 where $B=\max \{1, \vert E (\f_1)\vert,\vert E (\f_2)\vert, \vert E (\p)\vert \}$. 
 \end{thm}

\begin{proof}
 Set $\f := \sup \{\f_1,\f_2\}$. Observe that $\sup_X \f=\sup_X \f_1=\sup_X \f_2$
 and $\vert \f_1 - \f_2\vert = 2 (\f - \f_1) - (\f_2 - \f_1)$, thus
 $$
 \int_X \vert \f_1 - \f_2\vert \MA (\p) = 2 \int_X (\f - \f_1) \MA (\p) -  \int_X  (\f_2 - \f_1) \MA (\p).
 $$
 The second term on the right hand side is bounded from above by the desired quantity thanks to Lemma \ref{lem:corol}.
 
 We estimate the first one by using the same lemma, obtaining
 $$
 \int_X (\f - \f_1) \MA (\p) \leq C_n \cdot D^2 \cdot I (\f,\f_1)^{1 \slash 2^{n}},
 $$
 where $D:= \max \{1, \vert E (\f)\vert, \vert E (\f_1)\vert, \vert E (\p)\vert \}$.
 
Now $\vert E (\f) \vert \leq \vert E (\f_1) \vert$, since $0 \geq \f \geq \f_1$.
It therefore suffices to show that  $ I (\f,\f_1) \leq   I (\f_2,\f_1)$.
Recall that
$$
I (\f,\f_1) = \int_X (\f - \f_1) (\MA (\f_1) - \MA (\f)).
$$
and observe that  $\MA(\f) = \MA(\f_1)$ on the plurifine open set $\{\f_1 > \f_2\}$ (see \cite{BT87,GZ05,BEGZ}). 
Thus the measure $\MA (\f_1) - \MA (\f)$ is carried by the Borel set $\{\f_2  \geq \f_1\}$ where 
$\f - \f_1 = \f_2 - \f_1$.
Therefore
$$
I (\f,\f_1) = \int_X (\f_2 - \f_1) (\MA (\f_1) - \MA (\f)).
$$
In the same way we get
$$
I (\f,\f_2) = \int_X (\f_1 - \f_2) (\MA (\f_2) - \MA (\f)).
$$
Adding the two identities yields
$$
I (\f,\f_1) + I (\f,\f_2) = I (\f_1,\f_2),
$$
hence $I (\f,\f_1) \leq I (\f_1,\f_2)$.
 \end{proof}

 \begin{rem}
 We let the reader verify that Proposition \ref{thm:general} is a particular case of Theorem \ref{thm:B}.
 The latter has the following interesting consequence: if we let $\p$ be any $\theta$-psh function such that
 $V_{\theta}-1 \leq \p\leq V_{\theta}$ , then Chebyshev inequality, together with Theorem \ref{thm:B}, shows that for all $\e>0$,
 $$
 \rm{Cap}(\{|\f_1-\f_2|>\e\}) \leq \frac{C_n}{\e} B^2 I(\f_1,\f_2)^{2^{-n}}.
 $$
 This yields a quantitative estimate on how "convergence in energy" implies "convergence in capacity".
 \end{rem}

\section{Strong stability} \label{sec:strong}
  
Let $\mu = f_{\mu} \omega^n$ be a non-negative Radon measure wich is absolutely continuous with respect
 to a fixed volume form $\omega^n$, with density in $L^p$ for some $p>1$. When
 $\mu(X)=\vol(\a)$, it has been shown in \cite{BEGZ} that the complex Monge-Amp\`ere equation 
 $$
 \langle (\theta + dd^c \f_{\mu})^n \rangle = \mu=f_{\mu} \omega^n,
 $$
 has a unique solution $\f_{\mu} \in \cP(X,\theta)$ with minimal singularities such that $\sup_X \f = 0$. This is
 a generalization to the case of big cohomology classes of a celebrated result of Kolodziej \cite{Kol98}
 (which itself generalized Yau's celebrated ${\mathcal C}^0$ a priori estimate \cite{Yau78}).
 
 In this section we prove Theorem C of the introduction, establishing a quantitative continuity property of the mapping $f_{\mu} \mapsto \f_{\mu}$. Since measures with $L^p$ densities, $p>1$, satisfy
conditions ${\mathcal H}(\b)$ for all $\b>0$, Theorem C is actually a consequence of the following
more general result:


\begin{thm} \label{thm:strong}
Fix $\b>0$ and assume $\mu,\nu$ are non-negative Radon measures which satisfy the condition
${\mathcal H}(\b)$ and are normalized so that
$$
\mu(X)=\nu(X)=\vol(\a).
$$
Let $\f_{\mu},\f_{\nu}$ be their normalized Monge-Amp\`ere potentials. Then
$$
||\f_{\mu}-\f_{\nu}||_{L^{\infty}(X)} \leq M_{\tau} ||\mu-\nu||^{\tau}
$$
where $\tau=\g/(2^n-\gamma)$ with $\g := \b/[n+\b(n+1)]$.
\end{thm}

 When $\a$ is  a K\"ahler class, Theorem C is due to Kolodziej \cite{Kol03} who obtained a better exponent $\tau$ (see \cite{DZ}
 for a sharp improvement of the exponent).
 
 \smallskip
  
We need the following refinement of a statement proved in \cite{EGZ} in the context of big and semi-positive cohomology classes:

\begin{prop} 
Let $\nu$ be a non negative Radon measure which satisfies the condition $\mathcal H (\infty)$. 
Let $\mu = f \nu$, where $0 \leq f \in L^p (X, \nu)$ with $p > 1$ and $\mu(X)=\vol(\a)$.
Fix $\f,\p \in \cP(X,\theta)$ such that $\sup_X \f=\sup_X \p$ and $\MA (\f) = \mu$. 
 Then for any $0 < \gamma < \frac{1}{n q + 1}$, 
$$
\sup_X (\psi - \f)_+ \leq M \Vert(\p - \f)_+ \Vert_{L^1 (X, \nu)}^{\gamma},
$$
where $M > 0$  only depends on $\gamma$ and a bound on the $L^p-$norm of $f$.
\end{prop}

Here $u_+=\max(u,0)$ denotes as usual  the maximum of $u$ and $0$.

\smallskip

Let us stress that this relatively technical statement has interesting applications (see e.g \cite{DDGHKZ} where
it is used to establish H\"older-continuity properties of Monge-Amp\`ere potentials).
It is an immediate consequence of the following slightly more general (and more technical) result:
 
\begin{prop} \label{pro:EGZ}
Let $\f,\p$ be  $\theta$-plurisubharmonic functions such that 
$$
- M_0 + V_{\theta} \leq \sup \{\f , \psi\} \leq V_{\theta},
$$ 
for some $M_0 >0$. Assume that $\mu:=(\theta+dd^c \f)^n$  satisfies the condition 
$\mathcal H (\b)$ for some $\b>0$.
 Then  there exists $A_0 = A_0(\b,M_0)$ such that for any $r > 0$ we have
 $$
  \sup_X (\psi - \f)_+ \leq A_0 \Vert(\psi - \f)_+ \Vert_{L^r (\mu)}^{\gamma}
 \text{ with }
 \gamma = \frac{\b r }{n + \b (n + r)}.
 $$
Moreover if $\mu = f \nu,$ where $\nu$ a Borel measure and $f \in L^{p} (\nu)$, $p > 1$, then there exists $0< A_1 = A_1(\b,M_0,p)$ such that
 $$
  \sup_X (\psi - \f)_+ \leq A_1 \Vert f\Vert_{L^p (\nu)}^{\gamma q} \Vert(\psi - \f)_+ \Vert_{L^1 (\nu)}^{\gamma'},
 \text{ with }
 \g' = \frac{\b}{q n + \b (n q + 1)},
 $$
 where $1/p+1/q=1$ and $(\p-\f)_+:=\max(\p-\f,0)$.
\end{prop}
\smallskip
Although the proof is very close to that of Propositions 2.6 and 3.1 in \cite{EGZ}, we briefly
sketch it for the convenience of the reader.

\begin{proof}
Observe first that $ (\psi - \f)_+ = \sup \{\f,\psi\} - \f$ on $X$. So up to replacing $\psi$ by $\sup \{\f,\psi\}$, we can assume that $\psi \geq \f$ and $\psi$ satisfies the condition
$ -M_0 + V_{\theta} \leq \psi \leq V_{\theta}$ on $X$. 

Using the "big" comparison principle from \cite{BEGZ} and arguing exactly as in  Proposition 2.6 in \cite{EGZ}, we conclude that 
there is a constant $B_0 > 0$ such that for any $\e \in ]0,1]$
$$
\sup_X (\psi - \f) \leq \e + B_0 \left(Cap (\{\psi - \f > \e\}\right)^{\b \slash n}
$$
The proof of \cite[Proposition 2.6]{EGZ} (cf equation (3) p.616) shows that
$$
\e^n Cap (\{\psi - \f > \e\}) \leq (1 + M_0)^n \int_{ \{\psi - \f > \e\slash 2\} } d \mu.
$$
Chebyshev's inequality then yields
$$
 Cap (\{\psi - \f > \e\}) \leq 2^r \e^{- (n + r)} (1 + M_0)^n \int_X {{(\psi - \f)_+}^r} d \mu,
$$
for $r>0$ fixed.
Therefore 
$$
\sup_X (\psi - \f) \leq \e + B_0 2^{\b r \slash n}  (1 + M_0)^\a  \e^{- \b (n + r) \slash n}  \left(\int_X {{(\psi - \f)}_+}^r d \mu  \right)^{\b \slash n}
$$
Choosing $\e := (\Vert\psi - \f\Vert_{L^r (\mu)} \slash N)^{\g},$ where $N$ is an upper bound on $\psi - \f$ and $\g$ is
as in the satement of the proposition yields the desired inequality. 

 Now if $\mu = f \nu,$ where $f \in L^p (\nu)$ with $p>1$, H\"older's inequality yields
 $$
 \int_X {(\psi - \f)^r}_+ d \mu \leq \Vert f \Vert_{L^p (\nu)} \left(\int_X {(\psi - \f)}^{r q} d \nu \right)^{1 \slash q}.
 $$
The conclusion follows by taking $r := 1 \slash q$.
\end{proof}

\medskip

\noindent {\bf Proof of Theorem \ref{thm:strong}.}  
 Since $\f=\f_{\mu}$ and $\psi=\f_{\nu}$ have minimal singularities, $\f - \p$ is bounded hence
  \begin{eqnarray*}
  I(\f,\p) & = & \int_X(\f-\p)(\MA(\p)-\MA(\f)) = \int_X (\f - \p) d (\nu - \mu) \\
  & \leq & \Vert \f - \p\Vert_{L^{\infty}(X)} \Vert \mu - \nu\Vert
  \end{eqnarray*}
It follows from Proposition \ref{pro:EGZ} that
$$
\Vert \f - \p\Vert_{L^{\infty} (X)} \leq C_{\b} \left[ \Vert \f - \p\Vert_{L^1 (X,\mu)}^{\gamma}
+\Vert \f - \p\Vert_{L^1 (X,\nu)}^{\gamma} \right].
$$
with $\g:= \b/[n+\b(n+1)]$.
Now Theorem \ref{thm:B} implies
$$
\Vert \f - \p\Vert_{L^{\infty} (X)} \leq C_{\b}' \left(\Vert \f - \p\Vert_{L^{\infty}} \Vert \mu - \nu\Vert\right)^{\gamma \slash 2^n},
$$
thus
  $$
  \Vert \f - \p\Vert_{L^{\infty} (X)} \leq C_{\b}'' \Vert \mu - \nu\Vert^{\tau}
  $$
  where $\tau := \frac{\gamma}{2^n - \gamma}$.
 \hfill $\Box$

  \end{document}